\documentclass[11pt]{amsart}
\usepackage{amsmath}
\usepackage{amssymb}
\usepackage{graphicx}
\newtheorem{theorem}{Theorem}[section]

\newtheorem{lemma}[theorem]{Lemma}

\theoremstyle{definition}

\numberwithin{equation}{section}
\title[Bridge to Hyperbolic Polygonal Billiards]{Bridge to Hyperbolic Polygonal Billiards}
\author[H. Attarchi]{Hassan Attarchi}
\address[H. Attarchi]{School of Mathematics, Georgia Institute of Technology, Atlanta, US}
\email{{\tt hattarchi@gatech.edu}}
\author[L. A. Bunimovich]{Leonid A. Bunimovich}
\address[L. A. Bunimovich]{School of Mathematics, Georgia Institute of Technology, Atlanta, US}
\email{\tt leonid.bunimovich@math.gatech.edu}
\thanks{To Yakov Grigorievich Sinai on his 85 birthday}
\begin{document}
\begin{abstract}
It is well-known that billiards in polygons cannot be chaotic (hyperbolic). Particularly Kolmogorov-Sinai entropy of any polygonal billiard is zero. We consider physical polygonal billiards where a moving particle is a hard disc rather than a point (mathematical) particle and show that typical physical polygonal billiard is hyperbolic at least on a subset of positive measure and therefore has a positive Kolmogorov-Sinai entropy for any positive radius of the moving particle (provided that the particle is not so big that it cannot move within a polygon). This happens because a typical physical polygonal billiard is equivalent to a mathematical (point particle) semi-dispersing billiard. We also conjecture that in fact typical physical billiard in polygon is ergodic under the same conditions. 
\end{abstract}
\maketitle
\section{Introduction}
In the last years, there has been significant research on billiards in polygons \cite{BT12,BGKT98,FU14,FSU18,Gal83,ST15,Vor97}. Dynamics of these models is extremely difficult to rigorously analyze which often happens with systems with intermediate, neither regular (integrable) nor chaotic, behavior. Billiards were introduced and successfully used as models for numerous phenomena and processes in nature, especially in physics \cite{DCB99,GO74,HC69,Lor05,WL71}. However, only mathematical billiards were considered, where a point (mathematical) particle moves. There are no and there will be no such particles in reality. Nevertheless, studies of just mathematical billiards were considered to be sufficient. The reasons for that were twofold. First, a system with real (finite size) particles could be sometimes reduced to a mathematical (point particle) billiard in some peculiar billiard table. Such situation takes place for instance for celebrated Boltzmann gas of hard spheres \cite{Boltz}. Moreover, all basic examples of billiards with regular dynamics (e.g. billiards in circles and rectangles) have absolutely the same dynamics if one considers a (not too large) hard disc moving within the same billiard table. The same happens to the most popular chaotic (hyperbolic) billiards like Sinai billiards and squashes (stadium is a special case of a squash) \cite{Bun19,Sin70}.\par
It has been shown, however, that the transition to physical billiards can completely change the dynamics. Moreover, any type of chaos-order or order-chaos transition may occur \cite{Bun19}. In particular, it has been shown that classical Ehrenfests' Wind-Tree gas has richer dynamics than the Lorentz gas if a moving particle is real (physical) \cite{ABB}. In the present paper, we show that typical physical billiard in polygons is chaotic for an arbitrarily small size (radius) of a moving particle. The last means that physical billiards in generic polygons are hyperbolic on a subset of positive measure and, particularly, have a positive Kolmogorov-Sinai entropy to the contrary to mathematical billiards in polygons, which have zero KS-entropy.
\section{Billiards in Polygons}
    Let $\mathbf{P}$ be the space of all closed polygons in $\mathbf{R}^2$ and $\mathbf{P}_n\subset\mathbf{P}$ denote the space of all polygons with $n$ vertices. Let $\{v_0,\ v_1,\ \dots,\ v_{n-1}\}$ be the set of vertices of a polygon in $\mathbf{P}_n$. If we fix one side of this polygon on the $x$-axis and one of the vertices of that side at the origin (e.g. $v_0=(0,0)$ and $v_{n-1}=(x,0)$) then the embedding $\mathbf{P}_n\rightarrow\mathbf{R}^{2n-3}$ induces a topology in $\mathbf{P}_n$ such that its corresponding metric makes the space $\mathbf{P}_n$ complete \cite{ZK75}. If all angles of a polygon are commensurate with $\pi$, then it is called a rational polygon. It is well-known that rational polygons are dense in $\mathbf{P}$.\par
	A billiard in a polygon $P\in\mathbf{P}$ is a dynamical system generated by the motion of a point particle along a straight line inside $P$ with unit speed and elastic reflections off its boundary $\partial P$.\par
	The phase space of this dynamical system is
	$$\Lambda_P=\{(x,\varphi)\in\partial P\times(\frac{-\pi}{2},\frac{\pi}{2})\ :\ x\ is\ not\ a\ vertex\ of\ P\},$$
	where $\varphi$ is the reflection angle with respect to the inward normal vector $n(x)$ to the boundary at reflection point $x\in\partial P$.\par
	Let $\gamma$ be a billiard orbit in the polygon $P$. If the orbit $\gamma$ hits a side of the boundary $\partial P$ then instead of reflecting the orbit $\gamma$ off that side of $P$, one may reflect $P$ about that side. Denote the reflected polygon by $P_1$. The unfolded orbit $\gamma$ is a straight line as the continuation of $\gamma$ in $P_1$. In the geometric optics this procedure is called the method of images or unfolding \cite{Gut86}. Continuing this procedure for $n$ consecutive reflections of the orbit $\gamma$, we obtain a sequence of polygons $P,\ P_1,\ P_2, \dots,\ P_n$ where the unfolded orbit $\gamma$ is a straight segment through $P_1$ to $P_n$ (Fig. \ref{fig0}).\par
	\begin{figure}
	\centering
	    \includegraphics[width=7.5cm]{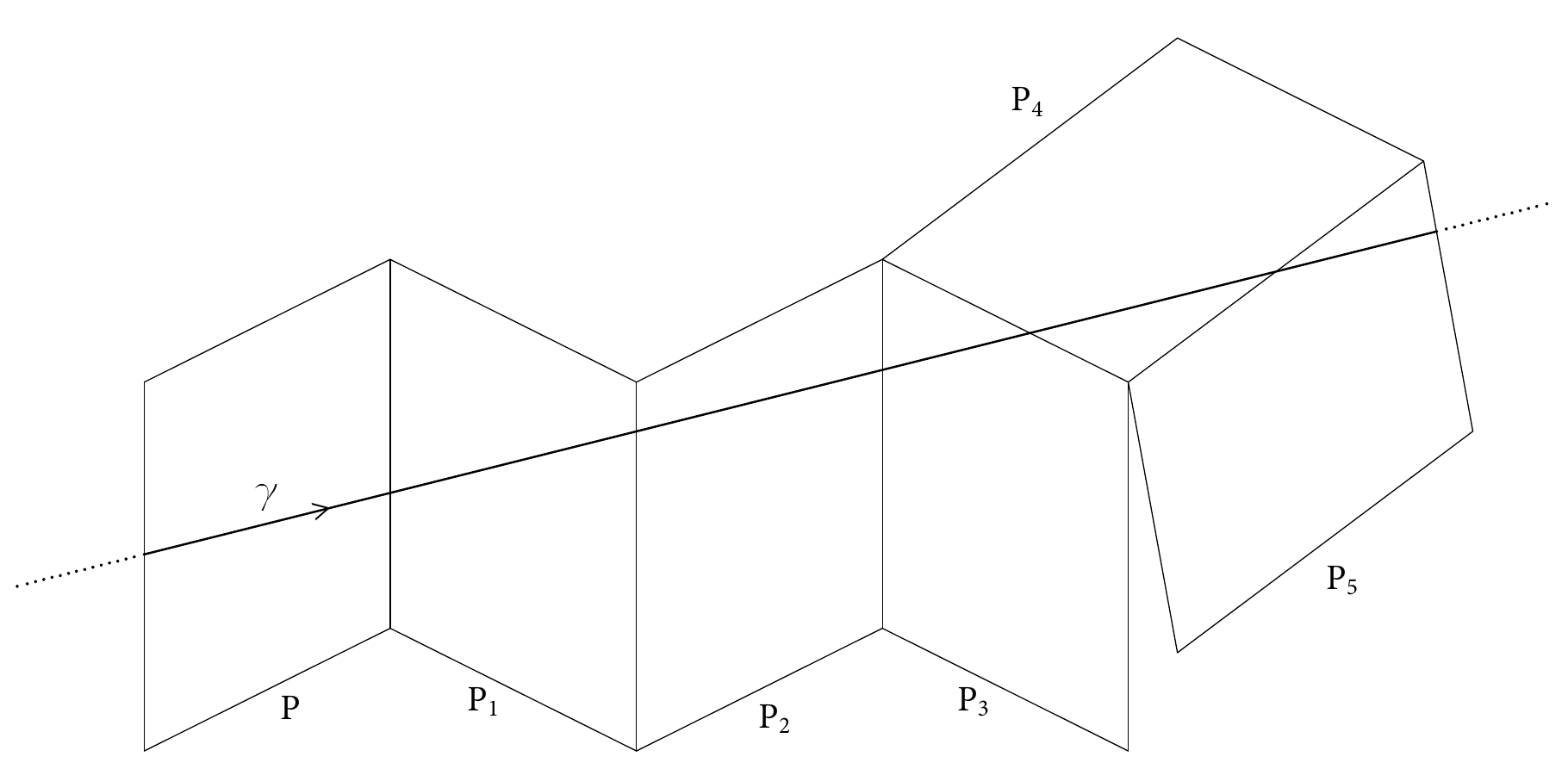}\caption{The unfolding process.}\label{fig0}
    \end{figure}
    The unfolding process can also be done backward in time. A trajectory stops when it hits a vertex. The unfolded orbit $\gamma$ is a finite segment if it hits vertices of $P$ both in the future and in the past. Such trajectories are called \emph{generalized diagonals}. In \cite{Gut86,ZK75}, it is shown that the set of generalized diagonals of the polygon $P\in\mathbf{P}$ is countable.\par
    Consider $(x,\varphi)\in\Lambda_P$. A direction $\varphi$ at point $x$ is called an \emph{exceptional direction} if its trajectory hits a vertex of $P$. It is not difficult to see that the number of these exceptional directions is countable at each point $x$.
\section{Physical Billiards in non-Convex Polygons}
	A physical billiard in a domain (billiard table) is generated by the motion of a hard ball (disc) of radius $r>0$ in that domain with unit speed and elastic reflections off the boundary. To represent the dynamics of such ball (a physical particle) of radius $r>0$, it is enough to follow the motion of its center. We can see that the center of particle moves within a smaller billiard table, which one gets by moving any point $x$ of the boundary by $r$ to the interior of the billiard table along the internal normal vector $n(x)$ \cite{Bun19}.\par
	It is easy to see that dynamics of a physical billiard in a convex simply connected polygon is completely equivalent to dynamics of a mathematical billiard in this polygon \cite{Bun19}. However, the situation is totally different for non-convex polygons (more precisely, polygons with at least one reflex angle, see Fig. \ref{fig1}). In this case, the boundary of the equivalent mathematical billiard acquires some dispersing parts, which are arcs of a circle of radius $r$ (see e.g. \cite{CF16,CFZ18}).\par
	\begin{figure}[h]
	\centering
	    \includegraphics[width=10cm]{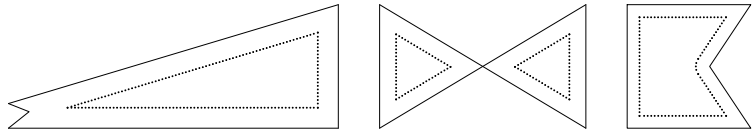}\caption{To have dispersing parts in the boundary of mathematical billiards equivalent to physical billiards in non-convex polygons, the particle has to be small enough and the polygon has to have at least one reflex angle.}\label{fig1}
    \end{figure}
    Let $\mathbf{P}_{ref}$ be the set of all polygons that they have at least one reflex angle.	To show that $\mathbf{P}_{ref}$ is dense in $\mathbf{P}$, we use the metric $d(.,.)$ on $\mathbf{P}$ which is defined as:
	\begin{equation}\label{e0}
	d(P,Q)=\int_{\mathbf{R}^2}|\chi_P(x)-\chi_Q(x)|dx,
	\end{equation}
	where $P,Q\in\mathbf{P}$ and,
	$$\chi_P(x)=\left\{\begin{array}{l}
	1\hspace{.5cm}x\in P,\\
	0\hspace{.5cm}otherwise.
	\end{array}\right.$$
	The topology induced by the metric $d(.,.)$ in $\mathbf{P}_n$ is equivalent to the induced topology in $\mathbf{P}_n$ by the embedding $\mathbf{P}_n\rightarrow\mathbf{R}^{2n-3}$. Thus, rational polygons are dense in $\mathbf{P}$ with respect to the metric $d(.,.)$.
    \begin{lemma}\label{lem1}
		$\mathbf{P}_{ref}$ is dense in $\mathbf{P}$.
	\end{lemma}
	\begin{proof}
	    Let $P\in\mathbf{P}$ be a polygon with $n$ vertices $\{v_0,\ v_1,\ \dots,\ v_{n-1}\}$. Without loss of generality, we assume that the randomly chosen edge of $P$ is $v_0v_{n-1}$. On the perpendicular bisector of $v_0v_{n-1}$ in the interior of $P$, we choose a sequence of points $\{v_{n_k}\}_{k=k_0}^\infty$ such that the reflex angle $\angle v_0v_{n_k}v_{n-1}=\pi+\frac{\pi}{k}$ ($k_0$ is big enough to have $v_{n_k}$ for $k\geq k_0$ in the interior of $P$, see Fig. \ref{fig15}).\par
	    \begin{figure}[h]
	    \centering
	        \includegraphics[width=6cm]{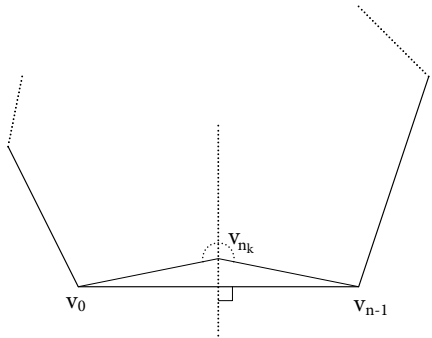}\caption{Replacing one edge of a polygon by a reflex angle.}\label{fig15}
       \end{figure}
	    If we denote the non-convex polygons with vertices $\{v_0,\ v_1,\ \dots,\ v_{n-1},\ v_{n_k}\}$ by $P_k$ then it follows that $P_k\in\mathbf{P}_{ref}$ and
	    $$d(P,P_k)\rightarrow0,$$
	    as $k\rightarrow\infty$.
	\end{proof}	
	\begin{lemma}\label{lem2}
	    $\mathbf{P}_{ref}$ is open in $\mathbf{P}$.
	\end{lemma}
	\begin{proof}
	    It is easy to see that any perturbation of $P\in\mathbf{P}_{ref}$ will have at least one reflex angle. This means $\mathbf{P}_{ref}$ is open in $\mathbf{P}$.
	\end{proof}
    Continued fractions for billiards were introduced in Sinai's fundamental paper \cite{Sin70}. They serve as a basic tool for analysis of billiards dynamics. Let $0=t_0<t_1<t_2<\dots$ be the reflection times of the trajectory $\gamma$ off the boundary $\partial Q$ where $Q$ is an arbitrary billiard table. Denote by $\kappa_i$ the curvature of the boundary at the $i$ \!\!th reflection point with respect to the inward unit normal vectors $n(x)$ to the boundary at $x\in\partial Q$, and by $\varphi_i$ the $i$ \!\!th reflection angle such that $\frac{-\pi}{2}<\varphi_i<\frac{\pi}{2}$. The corresponding continued fraction of this trajectory is given by
    $$\varkappa=\cfrac{1}{\tau_1+\cfrac{1}{\cfrac{2\kappa_1}{\cos\varphi_1}+\cfrac{1}{\tau_2+\cfrac{1}{\cfrac{2\kappa_2}{\cos\varphi_2}+\cfrac{1}{\tau_3+\cfrac{1}{\ddots}}}}}},$$
	where $\tau_i=t_i-t_{i-1}$ for $i=1,2,\dots$.\par
	Let $P\in\mathbf{P}_{ref}$, then the curvature of boundary components of the mathematical billiard equivalent to the physical billiard in $P$ is either $0$ or $\frac{1}{r}$. If an orbit hits the dispersing components infinitely many times where reflection numbers on dispersing parts are given by the sequence $\{i_k\}_{k=1}^\infty$, then the continued fraction of this orbit will have the following form between $i_j$ \!\!th and $i_{j+1}$ \!\!th reflections,
	$$\ddots+\cfrac{1}{\cfrac{2}{r\cos\varphi_{i_j}}+\cfrac{1}{(\tau_{i_j+1}+\dots+\tau_{i_{j+1}})+\cfrac{1}{\cfrac{2}{r\cos\varphi_{i_{j+1}}}+\cfrac{1}{\ddots}}}}.$$
	All elements of continued fractions in this case are positive. Also, almost any orbit has finitely many reflections within any finite time interval, since the boundary components are $C^\infty$ (they are line segments or arcs of a circle of radius $r$). Therefore,
	\begin{equation}\label{SS-thm}
	\sum_{k=0}^\infty\left((\tau_{i_k+1}+\dots+\tau_{i_{k+1}})+\frac{2}{r\cos\varphi_{i_{k+1}}}\right)=\infty,
	\end{equation}
	where $i_0=0$.\par
	Let $\mathbf{\hat{P}}$ denote the space of all non-convex simply connected rational polygons. Then, $\mathbf{\hat{P}}\subset\mathbf{P}_{ref}$.
	\begin{theorem}\label{thm1}
		For any $P\in\mathbf{\hat{P}}$, there exists $r_P>0$ such that the physical billiard in $P$ is hyperbolic for all $r<r_P$.
	\end{theorem}
	\begin{proof}
        Let $P\in\mathbf{\hat{P}}$ have $n$ vertices. Assume $\{v_0,\ v_1,\dots,\ v_{n-1}\}$ and $\{e_1=v_0v_1,\ e_2=v_1v_2,\dots,\ e_n=v_{n-1}v_0\}$ are sets of its vertices and edges, respectively. Let
        $$r_k=\min\{|v_k-x|\ :\ x\in e_i\ for\ i\neq k\ and\ i\neq k+1\},$$
        where $|.|$ is the euclidean distance in $\mathbf{R}^2$ and $k=0,1,\dots,n-1$ (if $k=0$ then $i=2,3,\dots,n-1$). It is easy to see that $r_k$ is well-defined since it is the minimum value of a continuous function on a compact set. Moreover, $r_k>0$. If we let 
        \begin{equation}\label{e1}
	        r_P=\min\{\frac{r_0}{2},\frac{r_1}{2},\dots,\frac{r_{n-1}}{2}\},
        \end{equation}
        then the hard ball of radius $r<r_P$ will be able to hit all edges of $P$. Therefore, when $r<r_P$ the boundary of the equivalent mathematical billiard has some dispersing components.\par
        In fact, if a radius of the particle is sufficiently large then some parts of the boundary of a billiard table become ``non-visible" to the particle. Therefore it does not matter for dynamics what is the exact structure of this ``non-visible" boundary. Such situation may occur e.g. for polygons \cite{BT04}.\par
        As long as a trajectory does not hit dispersing parts, it can be considered as a trajectory in the rational polygon $P'$ that shapes by replacing the dispersing components of the boundary by flat segments (Fig. \ref{fig2}). More precisely, angles $\theta$ and $\alpha$ satisfy the equation $\alpha=\frac{\pi+\theta}{2}$ and $\theta$ is commensurate with $\pi$, therefore, $\alpha$ is commensurate with $\pi$.\par
	\begin{figure}[h]
	\centering
	    \includegraphics[width=6cm]{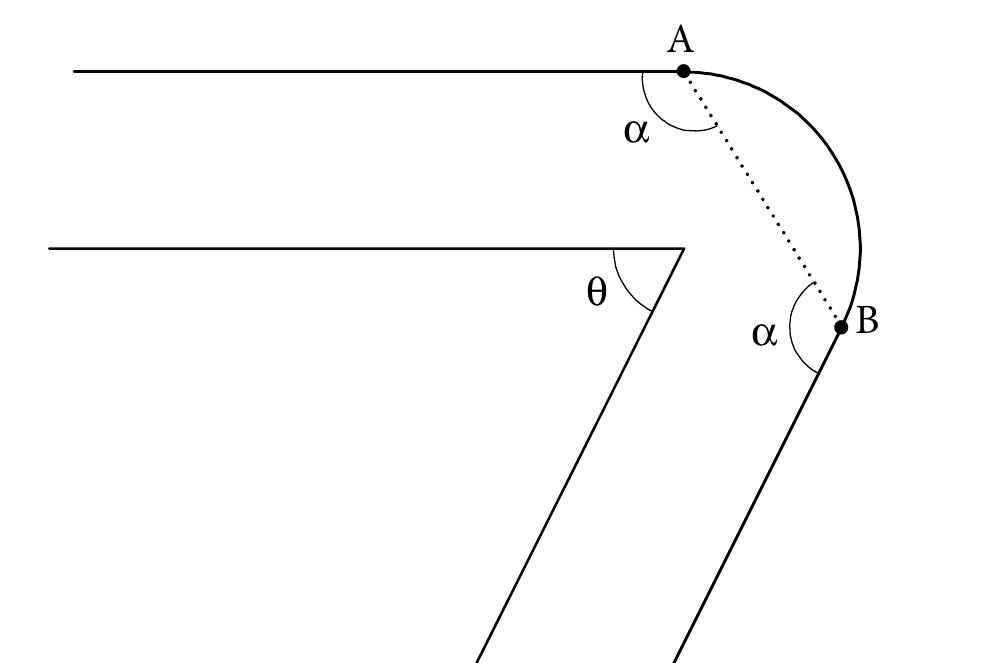}\caption{Replacing a dispersing part with a line segment.}\label{fig2}
    \end{figure}
        It is well-known that almost all orbits of billiards in rational polygons are spatially dense inside the billiard table \cite{BKM78,KMS86,ZK75}. Thus, all non-exceptional trajectories in $P'$ will hit all edges of $P'$, including those replaced the dispersing parts of the boundary. This implies that almost all trajectories (full measure in phase space) in the mathematical billiard equivalent to the physical billiard in $P$ will hit at least one dispersing component. Note that after the first reflection off a dispersing component, the forward trajectory is not the same as the one in $P'$. Thus, we cannot use density of almost all orbits in $P'$ to show that the trajectory will hit dispersing parts of the boundary infinitely many times.\par
        So, there is a full measure subset of points in the phase space such that their trajectories hit at least one dispersing component. Let $(x,\theta)$ be a point in that subset. By the continuity of the system on initial conditions, there is a neighborhood of positive measure of the point $(x,\theta)$ such that trajectories of all points in that neighborhood hit the same dispersing part as the trajectory of $(x,\theta)$ hits for the first time. Then the Poincare recurrence theorem implies that almost all trajectories in this neighborhood will return and hit that dispersing part infinitely many times. The convergence of continued fractions of such trajectories that hit dispersing components follows from the Seidel-Stern theorem and (\ref{SS-thm}). Hence, for a full measure subset of the phase space of the physical billiard in non-convex simply connected rational polygons, we have hyperbolicity. Moreover, it implies there are at most countable number of ergodic components such that the Kolmogorov-Sinai entropy is positive on each of them \cite{BS73,Bun90,Sin70}.
    \end{proof}
    It follows from Lemma \ref{lem1} and \ref{lem2} that $\mathbf{P}_{ref}$ is an open dense set in $\mathbf{P}$. Therefore, being a polygon with at least one reflex angle in $\mathbf{P}$ is topologically generic.
    \begin{theorem}
        There is an open dense subset of $\mathbf{P}$ such that the physical billiard in the polygons of this subset (when the radius of the hard ball is small enough) is hyperbolic on a subset of positive measure of their phase spaces.
	\end{theorem}
	\begin{proof}
        Let $P\in\mathbf{P}_{ref}$ and $r_P>0$ be the maximum radius of the hard ball defined in (\ref{e1}) such that the particle can hit all edges of $P$. Then the boundary of the mathematical billiard equivalent to the physical billiard in $P$ has some dispersing components which are arcs of a circle of radius $r<r_P$. Let, $$A=\{(x,\varphi)\ :\ For\ all\ x\ in\ a\ dispersing\ component\ and\ \varphi\in(\frac{-\pi}{2},\frac{\pi}{2})\}.$$
	    It follows from the definition of $A$ that it is a subset of positive measure in the phase space (one can exclude the exceptional directions which form a measure zero set in the phase space). The Poincare recurrence theorem implies that almost all points of $A$ will return to $A$ infinitely many times under the action of the billiard map. That means almost all trajectories of points in $A$ will hit a dispersing part infinitely many times. Then the convergence of continued fractions of trajectories of almost all points of $A$ follows from the Seidel-Stern theorem and (\ref{SS-thm}). Hence a physical billiard in a polygon with a reflex angle is hyperbolic at least on a subset of positive measure of its phase space.
	\end{proof}
	We conjecture that in fact generically a physical billiard in polygon is ergodic for any radius of a moving particle (which is of course not that large that the particle cannot move within a polygon). To prove our conjecture one instead needs to show that almost all orbits in a physical billiard in a polygon will eventually hit any segment which is a part of any side of a polygon. Hence a ``large" physical particle must hit all (rather than one) vertices of a polygon. For instance, if each vertex of a convex polygon gets replaced by a focusing arc it is possible to prove ergodicity \cite{CT98}. In this case the mechanism of defocusing \cite{Bun74} ensures hyperbolicity of such semi-focusing billiards on entire phase space. The current theory of billiards in polygons  establishes only that any billiard orbit in a polygon is either periodic or its closer contains just one vertex of this polygon, which is by far not enough. We are confident though that our conjecture holds.

\end{document}